\documentclass[11pt]{article}%
\usepackage{amsmath}
\usepackage{amsfonts}
\usepackage{amssymb}
\usepackage{amsthm}
\usepackage{graphicx}
\providecommand{\U}[1]{\protect\rule{.1in}{.1in}}
\newtheorem{theorem}{Theorem}[section]

\newtheorem{corollary}[theorem]{Corollary}
\newtheorem{lemma}[theorem]{Lemma}
\newtheorem{proposition}[theorem]{Proposition}

\theoremstyle{definition}

\theoremstyle{remark}
\newtheorem{example}[theorem]{Example}
\newtheorem{remark}[theorem]{Remark}

\numberwithin{equation}{section}
\begin{document}

\title{Spectral parameter power series representation for solutions of linear system of two first order differential equations}
\author{Nelson Guti\'{e}rrez Jim\'{e}nez$^a$ and Sergii M. Torba$^b$\\
\\{\small $~^a$ Instituto de Matem\'{a}ticas, Facultad de Ciencias Exactas y Naturales, Universidad de Antioquia,}\\{\small Calle 67 No.~53--108, Medell\'{i}n, COLOMBIA}
\\{\small $~^b$ Departamento de Matem\'{a}ticas, CINVESTAV del IPN, Unidad
Quer\'{e}taro, }\\{\small Libramiento Norponiente No. 2000, Fracc. Real de Juriquilla,
Quer\'{e}taro, Qro. 76230, MEXICO}\\{\small e-mails:
nelson.gutierrez@udea.edu.co, storba@math.cinvestav.edu.mx}}
\maketitle

\begin{abstract}
A representation in the form of spectral parameter power series (SPPS) is given for a general solution of a one dimension Dirac system containing arbitrary matrix coefficient at the spectral parameter,
\begin{equation*}
B \frac{dY}{dx} + P(x)Y = \lambda R(x)Y,\eqno{(\ast)}
\end{equation*}
where $Y=(y_1,y_2)^T$  is the unknown vector-function, $\lambda$ is the spectral parameter, $B = \begin{pmatrix}
    0 & 1 \\
    -1 & 0
\end{pmatrix}$,
and $P$ is a symmetric $2\times 2$ matrix, $R$ is an arbitrary $2\times 2$ matrix whose entries are integrable complex-valued functions. The coefficient functions in these series are obtained by recursively iterating a simple
integration process, beginning with a non-vanishing solution for one particular $\lambda = \lambda_0$. The existence of such solution is shown.

For a general linear system of two first order differential equations
\begin{equation*}
    P(x)\frac{dY}{dx}+Q(x)Y = \lambda R(x)Y,\qquad x\in [a,b],
\end{equation*}
where $P$, $Q$, $R$ are $2\times 2$ matrices whose entries are integrable complex-valued functions, $P$ being invertible for every $x$, a transformation reducing it to a system ($\ast$) is shown.

The general scheme of application of the SPPS representation to the solution of initial value and spectral problems as well as numerical illustrations are provided.
\end{abstract}

\section{Introduction}
The spectral parameter power series (SPPS) representation for solutions of second-order linear differential equations  \cite{KrCV08}, \cite{KrPorter2010} has proven to be an efficient tool for solving (analytically and numerically) and studying a variety of problems, see the review \cite{KKRosu} and recent papers \cite{BR2016}, \cite{BR2017}, \cite{BRM2019}, \cite{CKT2015},  \cite{LO2017}, \cite{RH2017}. The SPPS method starts with a non-vanishing solution of the equation for one fixed value of the spectral parameter and by performing a series of recursive integrations produces coefficients of the Taylor series of the solution with respect to the spectral parameter. The procedure can be easily and efficiently implemented numerically allowing one to solve a variety of spectral problems with remarkable accuracy.

Later the SPPS representation was extended to solutions of singular second order differential equations \cite{CKT2013}, of equations with polynomial dependence on the spectral parameter \cite{KTV} and recently of linear differential equations of arbitrary order \cite{KrPoTo2018}. However the SPPS representation for the solutions of linear systems of differential equations was constructed only for Zakharov-Shabat system  and a particular case of one dimensional Dirac system \cite{KTV}, in both cases by transforming the system into a certain Sturm-Liouville equation. Even though the general one dimensional Dirac system can be transformed into a Sturm-Liouville equation with potential polynomially dependent on the spectral parameter, see \cite{AAS2012}, and the result from \cite{KTV} may be applied, we are not aware of such transformation for more complicated right-hand sides of the system. We opted for a different approach which allowed us to deal with arbitrary linear systems of two first order differential equations.

It is worth adding that the SPPS representation of solutions of one dimension Schr\"{o}dinger equations allowed us to establish the mapping theorem for transmutation operators \cite{CKrTr2012} which finally has led us to the development of two new methods, analytic approximation of transmutation operators (AATO) \cite{KT AnalyticApprox} and Neumann series of Bessel functions (NSBF) representation for the solutions \cite{KNT2017} allowing one, in particular, to obtain hundreds of approximate eigenvalues with non deteriorating accuracy. The representation proposed in this paper opens possibility to extend the AATO and NSBF methods onto one-dimensional Dirac systems.

The paper is organized as follows. In Section \ref{Section DiracSPPS} we consider a one-dimensional Dirac system whose right-hand side may contain arbitrary matrix-function coefficient at the spectral parameter. In Subsection \ref{SubSectFormalPowers} we introduce the formal powers starting from one non-vanishing particular solution corresponding to zero value of the spectral parameter. In Subsection \ref{SubSectDiracSPPS} we show how the general solution of the system can be written in the terms of these formal powers (Theorem \ref{Thm DiracSPPS}). In Subsection \ref{SubSectPartSol} we show the existence of a non-vanishing particular solution. In Subsection \ref{SubSectSpectralShift} we describe the spectral shift procedure. In Subsection \ref{SubSectDiscontinuous} we extend the results onto discontinuos coefficients. In Section \ref{SectionSPPSsystem} we show that by simple transformation the general linear system of two first order differential equations can be reduced to the form covered in Section \ref{Section DiracSPPS}. In Section \ref{Section Numerics} we propose the general scheme for application of the SPPS representation to the solution of initial value and spectral problems, present numerical results for a particular Dirac system,
show how a Sturm-Liouville spectral problem can be reformulated as a spectral problem for a Dirac system and discuss possible advantages of such problem reformulation.

\section{The spectral parameter power series representation for solutions of generalized Dirac systems}\label{Section DiracSPPS}
Consider the following system
\begin{equation}\label{GenDirac}
    \begin{cases}
        v'+p_1(x) u + q(x) v = \lambda \bigl( r_{11}(x) u + r_{12}(x) v\bigr),\\
        -u'+q(x) u + p_2(x) v = \lambda \bigl( r_{21}(x) u + r_{22}(x) v\bigr),
    \end{cases}
\end{equation}
or in the matrix form,
\begin{equation}\label{GenDiracMatrix}
B \frac{dY}{dx} + P(x)Y = \lambda R(x)Y,\qquad Y(x) = \begin{pmatrix}
    u(x) \\
    v(x) \\
\end{pmatrix},
\end{equation}
where
\[
B = \begin{pmatrix}
    0 & 1 \\
    -1 & 0 \\
\end{pmatrix},\qquad P(x) =
\begin{pmatrix}
    p_{1}(x) & q(x) \\
    q(x) & p_{2}(x) \\
\end{pmatrix},\qquad R(x)=
\begin{pmatrix}
    r_{11}(x) & r_{12}(x) \\
    r_{21}(x) & r_{22}(x) \\
\end{pmatrix},
\]
$p_i, q, r_{ij}\in C[a,b]$, $i,j\in\{1,2\}$ are complex-valued functions of the real variable $x$, and $\lambda$ is an arbitrary complex constant. In the case when
$R(x) = \begin{pmatrix}
    1 & 0 \\
    0 & 1 \\
\end{pmatrix}$ the system \eqref{GenDiracMatrix} is known as one-dimensional Dirac system \cite{LevitanSargsjan}, and in the case when $p_1=\bar p_2$, $q=0$ and $\begin{pmatrix}
    r_{11} & r_{12} \\
    r_{21} & r_{22} \\
\end{pmatrix} = \begin{pmatrix}
    0 & i \\
    i & 0 \\
\end{pmatrix}$ the system \eqref{GenDirac} is known as Zakharov-Shabat system \cite{Ablowitz y Segur, Zakharov-Shabat}.

\subsection{A system of generalized formal powers}\label{SubSectFormalPowers}
Suppose that the homogeneous system
\begin{equation}\label{HomDirac}
B \frac{dY}{dx} + P(x)Y = 0
\end{equation}
possesses a solution $Y=(f,g)^T$ such that both functions $f$ and $g$ are non-vanishing on $[a,b]$. From now on we will call such solution as non-vanishing solution of the homogeneous system \eqref{HomDirac}. Let $x_0$ be a point from the segment $[a,b]$. Consider the following systems of functions defined by the recursive relations
\begin{align}
    X^{(0)}(x) &= f(x_0) g(x_0) \int_{x_0}^x \frac{p_2(s)}{f^2(s)}\,ds, \label{X0}\\
    Y^{(0)}(x) &= 1 + f(x_0) g(x_0) \int_{x_0}^x \frac{p_1(s)}{g^2(s)}\,ds, \label{Y0}\\
    Z^{(n)}(x) &= \int_{x_0}^x \Bigl( X^{(n)}(s) \bigl(f^2(s) r_{11}(s) + g^2(s) r_{21}(s)\bigr) + Y^{(n)}(s) \bigl(f^2(s) r_{12}(s) + g^2(s) r_{22}(s)\bigr)\Bigr)\,ds, \label{Zn}\\
    X^{(n+1)}(x) &= (n+1)\int_{x_0}^x \Bigl( -r_{21}(s) X^{(n)}(s) - r_{22}(s)\frac{g(s)}{f(s)} Y^{(n)}(s)+\frac{p_2(s)}{f^2(s)}Z^{(n)}(s)\Bigr)\,ds,\label{Xn}\\
    Y^{(n+1)}(x) &= (n+1)\int_{x_0}^x \Bigl( r_{11}(s)\frac{f(s)}{g(s)} X^{(n)}(s) + r_{12}(s) Y^{(n)}(s)+\frac{p_1(s)}{g^2(s)}Z^{(n)}(s)\Bigr)\,ds,\quad n\ge 0.\label{Yn}
\end{align}
Similarly we use as the initial functions
\begin{align}
    \widetilde X^{(0)}(x) &= 1-f(x_0) g(x_0) \int_{x_0}^x \frac{p_2(s)}{f^2(s)}\,ds, \label{Xt0}\\
    \widetilde Y^{(0)}(x) &= -f(x_0) g(x_0) \int_{x_0}^x \frac{p_1(s)}{g^2(s)}\,ds \label{Yt0}
\end{align}
and define functions $\widetilde Z^{(n)}$, $\widetilde X^{(n+1)}$ and $\widetilde Y^{(n+1)}$, $n\ge 0$ using the same formulas \eqref{Zn}--\eqref{Yn} changing correspondingly all the functions $X^{(n)}$, $Y^{(n)}$ and $Z^{(n)}$ by $\widetilde X^{(n)}$, $\widetilde Y^{(n)}$ and $\widetilde Z^{(n)}$.

\begin{example}\label{ExampleSL}
Following \cite{KrPorter2010}, let us consider a Sturm-Liouville equation
\begin{equation}\label{SLeq}
\bigl(p(x)u'\bigr)'+q(x)u = \omega^2 r(x) u,
\end{equation}
where $p\in C^1[a,b]$, $q,r\in C[a,b]$ are complex-valued functions such that $p$ does not vanish on $[a,b]$. Suppose that a function $u_0$ is a non-vanishing solution corresponding to $\omega=0$. The following systems of functions were introduced in \cite{KrPorter2010}
\begin{align*}
\mathcal{X}^{(0)}(x)&\equiv \widetilde{\mathcal{X}}^{(0)} \equiv 1,\\
\displaybreak[2]
\mathcal{X}^{(n)}(x)&=
\begin{cases}
n\int_{x_0}^{x}\mathcal{X}^{(n-1)}(s)u_0^{2}(s)r(s)\,\mathrm{d}s, &n\text{ even}\\
n\int_{x_0}^{x}\mathcal{X}^{(n-1)}(s)\frac{1}{u_0^{2}(s)p(s)}\,\mathrm{d}s, &n\text{ odd}
\end{cases}\\
\displaybreak[2]
\widetilde{\mathcal{X}}^{(n)}(x)&=
\begin{cases}
n\int_{x_0}^{x}\widetilde{\mathcal{X}}^{(n-1)}(s)u_0^{2}(s)r(s)\,\mathrm{d}s, & n\text{ odd}\\
n\int_{x_0}^{x}\widetilde{\mathcal{X}}^{(n-1)}(s)\frac{1}{u_0^{2}(s)p(s)}\,\mathrm{d}s, & n\text{ even}
\end{cases}
\end{align*}
and it was proved that the general solution of \eqref{SLeq} has the form
\begin{equation}\label{Eq SPPSKravPorter}
u = c_1 u_0\sum_{k=0}^\infty \frac{\omega^{2k} \widetilde{\mathcal{X}}^{(2k)}}{(2k)!}+
c_2 u_0\sum_{k=0}^\infty \frac{\omega^{2k} \mathcal{X}^{(2k+1)}}{(2k+1)!}.
\end{equation}
Let $\omega \ne 0$. Consider a function $v$ defined by
\begin{equation}\label{EqSLtoDirac}
\omega \frac{v}{p} = u_0\frac{d}{dx}\left(\frac u{u_0}\right) = u'-\frac{u_0'}{u_0}u.
\end{equation}
Then equation \eqref{SLeq} is equivalent to the system
\begin{equation}\label{SLequivDirac}
\begin{cases}
v'+\frac{u_0'}{u_0}v = \omega ru,\\
-u'+\frac{u_0'}{u_0}u = -\omega \frac{1}{p}v,
\end{cases}
\end{equation}
a particular case of \eqref{GenDirac} having $q=u_0'/u_0$ and $p_1=p_2=0$. A non-vanishing particular solution of \eqref{SLequivDirac} corresponding to $\omega =0$ can be taken in the form $(u,v)^T = (u_0, 1/u_0)^T$ and one can easily verify that the formal powers defined by \eqref{X0}--\eqref{Yt0} satisfy for all $n\ge 0$
\begin{gather*}
X^{(2n)}=Y^{(2n+1)} \equiv 0,\qquad X^{(2n-1)} = \mathcal{X}^{(2n-1)},\qquad Y^{(2n)} =  \mathcal{X}^{(2n)},\\
\widetilde X^{(2n+1)}=\widetilde Y^{(2n)} \equiv 0,\qquad \widetilde X^{(2n)} =  \widetilde {\mathcal{X}}^{(2n)},\qquad \widetilde Y^{(2n+1)} = \widetilde {\mathcal{X}}^{(2n+1)}.
\end{gather*}
\end{example}

\subsection{The SPPS representation}\label{SubSectDiracSPPS}
The following theorem gives the general solution of the system \eqref{GenDiracMatrix}.
\begin{theorem}\label{Thm DiracSPPS}
Suppose that the homogeneous system \eqref{HomDirac} possesses a solution $Y_0=(f,g)^T$ such that both functions $f$ and $g$ are non-vanishing on $[a,b]$. Then a general solution of the system \eqref{GenDiracMatrix} has the form
\begin{equation}\label{DiracSPPSgen}
Y = c_1 Y_1 + c_2 Y_2 = c_1\begin{pmatrix}
    u_1 \\
    v_1 \\
\end{pmatrix} + c_2\begin{pmatrix}
    u_2 \\
    v_2 \\
\end{pmatrix},
\end{equation}
where $c_1$ and $c_2$ are arbitrary complex constants and
\begin{equation}\label{DiracSPPS}
    \begin{pmatrix}
    u_1 \\
    v_1 \\
\end{pmatrix} = \sum_{n=0}^\infty \frac{\lambda^n}{n!} \begin{pmatrix}
    f \widetilde X^{(n)} \\
    g \widetilde Y^{(n)} \\
\end{pmatrix},\qquad
\begin{pmatrix}
    u_2 \\
    v_2 \\
\end{pmatrix} = \sum_{n=0}^\infty \frac{\lambda^n}{n!} \begin{pmatrix}
    f X^{(n)} \\
    g Y^{(n)} \\
\end{pmatrix}.
\end{equation}
Here the formal powers $X^{(n)}$, $Y^{(n)}$, $\widetilde X^{(n)}$, $\widetilde Y^{(n)}$ are constructed starting with the particular solution $Y_0$ according to \eqref{X0}--\eqref{Yt0}.
Both series converge uniformly on $[a,b]$. The solutions $Y_{1}$ and $Y_{2}$ satisfy the following initial conditions:
\begin{equation}\label{SPPS IC}
Y_{1}(x_{0})=\begin{pmatrix}f(x_0)\\
                                           0
                             \end{pmatrix},
\qquad
Y_{2}(x_{0})=\begin{pmatrix}0\\
                                              g(x_0)
                               \end{pmatrix}.														 
\end{equation}
\end{theorem}

The proof of Theorem \ref{Thm DiracSPPS} requires several lemmas. In the first lemma we consider the nonhomogeneous system obtained from the left-hand side of \eqref{GenDirac} and construct a right-inverse operator for this system.

\begin{lemma}\label{Lemma NonHom Dirac}
Under the conditions of Theorem \ref{Thm DiracSPPS} the solution of the nonhomogeneous system
\begin{equation}\label{NonHom Dirac}
    \begin{cases}
        v'+p_1 u + q v = h_1,\\
        -u'+q u + p_2 v = h_2,
    \end{cases}
\end{equation}
with the initial conditions
\begin{equation}\label{NonHom Dirac IC}
    u(x_0) = v(x_0) = 0,
\end{equation}
where $h_{1,2}\in C[a,b]$ and $x_0\in [a,b]$, can be written in the form
\begin{equation}\label{NonHom Dirac u}
    u(x) = f(x) \int_{x_0}^x \biggl(-\frac{h_2(t)}{f(t)} + \frac{p_2(t)}{f^2(t)}\int_{x_0}^t \bigl( f(s) h_1(s) + g(s) h_2(s)\bigr)\,ds\biggr)dt
\end{equation}
and
\begin{equation}\label{NonHom Dirac v}
    v(x) = g(x) \int_{x_0}^x \biggl(\frac{h_1(t)}{g(t)} + \frac{p_1(t)}{g^2(t)}\int_{x_0}^t \bigl( f(s) h_1(s) + g(s) h_2(s)\bigr)\,ds\biggr)dt.
\end{equation}
\end{lemma}

\begin{proof}
First we show how one can obtain the formulas \eqref{NonHom Dirac u} and \eqref{NonHom Dirac v} in the case when $h_{1,2}$, $q$, $p_{1,2}\in C^1[a,b]$ and the functions $p_1$ and $p_2$ are non-vanishing on $[a,b]$.

By differentiation and simple algebraic transformations one can verify that if the functions $u$ and $v$ satisfy \eqref{NonHom Dirac}, then these functions satisfy the following nonhomogeneous Sturm-Liouville equations
\begin{align}
\left(\frac{u'}{p_2}\right)'+\left(p_1-\left(\frac{q}{p_2}\right)'-\frac{q^2}{p_2}\right) u&= h_1 - \frac q{p_2}h_2 - \left(\frac{h_2}{p_2}\right)',\label{NonHom SL1}\\
\left(\frac{v'}{p_1}\right)'+\left(p_2+\left(\frac{q}{p_1}\right)'-\frac{q^2}{p_1}\right) v& = h_2 - \frac q{p_1}h_1 + \left(\frac{h_1}{p_1}\right)'.\label{NonHom SL2}
\end{align}

Note that since the pair of functions $f$ and $g$ is a solution of the homogeneous system \eqref{HomDirac}, the function $f$ is a solution of homogeneous part of equation \eqref{NonHom SL1}, i.e., $\left(\frac{f'}{p_2}\right)'+\Bigl(p_1-\bigl(\frac{q}{p_2}\bigr)'-\frac{q^2}{p_2}\Bigr) f=0$, and $g$ is a solution of the equation $\left(\frac{g'}{p_1}\right)'+\Bigl(p_2+\bigl(\frac{q}{p_1}\bigr)'-\frac{q^2}{p_1}\Bigr)g=0$. We recall that if for the operator $L=\frac{d}{dx}p\frac{d}{dx}+q$ a non-vanishing function $u_0$ satisfying $Lu_0=0$ is known, then the operator  $L$ possesses the Polya's factorization $L=\frac 1{u_0}\partial pu_0^2 \partial \frac 1{u_0}$, see, e.g., \cite{Polya1924}. Using the Polya's factorization  and the Abel's formula, general solutions of equations \eqref{NonHom SL1} and \eqref{NonHom SL2} can be written in the form
\begin{multline}\label{NonHom SL1 GenSol}
    u(x) = f(x) \int_{x_0}^x \frac{p_2(t)}{f^2(t)}\int_{x_0}^t f(s)\left( h_1(s) - \frac{q(s)}{p_2(s)}h_2(s) - \left(\frac{h_2(s)}{p_2(s)}\right)'\right)\,ds\,dt\\ + c_{11}f(x) + c_{12}f(x)\int_{x_0}^x\frac {p_2(s)}{f^2(s)}\,ds
\end{multline}
and
\begin{multline}\label{NonHom SL2 GenSol}
    v(x) = g(x) \int_{x_0}^x \frac{p_1(t)}{g^2(t)}\int_{x_0}^t g(s)\left( h_2(s) - \frac{q(s)}{p_1(s)}h_1(s) + \left(\frac{h_1(s)}{p_1(s)}\right)'\right)\,ds\,dt\\ + c_{21}g(x) + c_{22}g(x)\int_{x_0}^x\frac {p_1(s)}{g^2(s)}\,ds.
\end{multline}

A pair $u$, $v$ of solutions of equations \eqref{NonHom SL1}, \eqref{NonHom SL2} is a solution of the Cauchy problem \eqref{NonHom Dirac IC} for the nonhomogeneous Dirac system \eqref{NonHom Dirac} if and only if the functions $u$, $v$ satisfy
\begin{equation}
    u(x_0) = v(x_0) = 0,\qquad u'(x_0) = -h_2(x_0),\qquad v'(x_0)=h_1(x_0). \label{NonHom SL IC}
\end{equation}
Hence we obtain from \eqref{NonHom SL1 GenSol}--\eqref{NonHom SL IC} that $c_{11}=c_{21}=0$,
\begin{equation}\label{NonHom SL PartSol}
    c_{12} = -\frac{h_2(x_0)f(x_0)}{p_2(x_0)}\qquad\text{and}\qquad c_{22}=\frac{h_1(x_0)g(x_0)}{p_1(x_0)}.
\end{equation}

Now we integrate by parts the expression $f\cdot \Bigl(\frac{h_2}{p_2}\Bigr)'$ in \eqref{NonHom SL1 GenSol} and use the equality $f'=qf + p_2g$ to obtain the formula \eqref{NonHom Dirac u}:
\begin{equation*}
    \begin{split}
      u(x) &= f(x)\int_{x_0}^x \frac{p_2(t)}{f^2(t)} \biggl[\int_{x_0}^t \left( f(s) h_1(s) - \frac{q(s)f(s)h_2(s)}{p_2(s)} + \frac{f'(s)h_2(s)}{p_2(s)}\right)\,ds \\
        & \quad -\frac{f(t)h_2(t)}{p_2(t)}+\frac{f(x_0)h_2(x_0)}{p_2(x_0)}\biggr]\,dt - \frac{h_2(x_0)f(x_0)f(x)}{p_2(x_0)}\int_{x_0}^x\frac{p_2(s)}{f^2(s)}\,ds\\
    &= f(x) \int_{x_0}^x \biggl(-\frac{h_2(t)}{f(t)} + \frac{p_2(t)}{f^2(t)}\int_{x_0}^t \bigl( f(s) h_1(s) + g(s) h_2(s)\bigr)\,ds\biggr)dt.
    \end{split}
\end{equation*}
Similarly one obtains the second formula \eqref{NonHom Dirac v}.

The validity of the formulas \eqref{NonHom Dirac u} and \eqref{NonHom Dirac v} without any additional requirements on the functions $q$, $p_{1,2}$ and $h_{1,2}$ can be checked directly. Indeed, using $g'=-p_1 f - qg$ we obtain
\begin{equation}\label{DiracEq1}
    \begin{split}
        v'(x)&+p_1(x)u(x)+q(x)v(x) = \frac{g'(x)}{g(x)} v(x) + h_1(x)\\
         &\quad + \frac{p_1(x)}{g(x)} \int_{x_0}^x \bigl(f(s)h_1(s)+g(s)h_2(s)\bigr)\,ds + p_1(x)u(x) + q(x)v(x) \\
        &= h_1(x) + \frac{p_1(x)}{g(x)}\biggl( g(x)u(x) - f(x) v(x) + \int_{x_0}^x \bigl(f(s)h_1(s)+g(s)h_2(s)\bigr)\,ds \biggr).
    \end{split}
\end{equation}
It follows from the formulas \eqref{NonHom Dirac u} and \eqref{NonHom Dirac v} that
\begin{equation}\label{DiracEq2}
    \begin{split}
    g(x)u(x)-f(x)v(x) &= f(x)g(x) \int_{x_0}^x \biggl[-\frac{h_1(t)}{g(t)}-\frac{h_2(t)}{f(t)} \\
    &\quad + \left(\frac{p_2(t)}{f^2(t)}-\frac{p_1(t)}{g^2(t)}\right)\int_{x_0}^t \bigl( f(s) h_1(s) + g(s) h_2(s)\bigr)\,ds\biggr]dt.
    \end{split}
\end{equation}
Since the functions $f$ and $g$ satisfy \eqref{HomDirac}, by simple algebraic manipulations we find that
\begin{equation}\label{DiracEq3}
    \left(\frac{p_2}{f^2}-\frac{p_1}{g^2}\right) = -\left(\frac 1{fg}\right)',
\end{equation}
and integrating by parts in \eqref{DiracEq2} we obtain
\begin{equation*}
    \begin{split}
    g(x)u(x)-f(x)v(x) &= -\int_{x_0}^x \bigl( f(s) h_1(s) + g(s) h_2(s)\bigr)\,ds\\
    & \quad + f(x)g(x) \int_{x_0}^x \left[-\frac{h_1(t)}{g(t)}-\frac{h_2(t)}{f(t)} + \frac{f(t)h_1(t) + g(t)h_2(t)}{f(t)g(t)} \right]\,dt\\
    &= -\int_{x_0}^x \bigl( f(s) h_1(s) + g(s) h_2(s)\bigr)\,ds.
    \end{split}
\end{equation*}
Hence the expression in the brackets in \eqref{DiracEq1} is equal to zero and we have verified the first equation in \eqref{NonHom Dirac}. Second equation can be verified similarly. The initial conditions \eqref{NonHom Dirac IC} follow directly from the formulas \eqref{NonHom Dirac u} and \eqref{NonHom Dirac v}.
\end{proof}

The next lemma in an analogue of the Abel's formula. We construct the general solution of the homogeneous system \eqref{HomDirac} starting from a known particular solution.
\begin{lemma}\label{Lemma Hom Dirac}
Under the conditions of Theorem \ref{Thm DiracSPPS} the general solution of the homogeneous system \eqref{HomDirac} has the form
\begin{equation}\label{GenSol Hom Dirac}
    \begin{pmatrix}
    u(x) \\
    v(x)\\
    \end{pmatrix} = c_1
    \begin{pmatrix}
    f(x)\left(1-\kappa \int_{x_0}^x \frac{p_2(s)}{f^2(s)}\,ds\right) \\
    -\kappa g(x) \int_{x_0}^x \frac{p_1(s)}{g^2(s)}\,ds\\
    \end{pmatrix} + c_2
    \begin{pmatrix}
    \kappa f(x) \int_{x_0}^x \frac{p_2(s)}{f^2(s)}\,ds \\
    g(x)\left(1+\kappa \int_{x_0}^x \frac{p_1(s)}{g^2(s)}\,ds\right)\\
    \end{pmatrix},
\end{equation}
where $\kappa:=f(x_0)g(x_0)$ and $c_1$ and $c_2$ are arbitrary complex constants.
\end{lemma}
\begin{proof}
Similarly to the proof of Lemma \ref{Lemma NonHom Dirac} under the additional assumptions on $p_{1,2}$ and $q$ the formula \eqref{GenSol Hom Dirac} can be obtained from the expressions \eqref{NonHom SL1 GenSol} and \eqref{NonHom SL2 GenSol} considering $h_1=h_2=0$ and two initial conditions $\begin{pmatrix}
u(x_0)\\
v(x_0)\\
\end{pmatrix}=\begin{pmatrix}f(x_0)\\
0\\
\end{pmatrix}$ (giving us the expression at $c_1$) and $\begin{pmatrix}
u(x_0)\\
v(x_0)\\
\end{pmatrix}=\begin{pmatrix}0\\
g(x_0)\\
\end{pmatrix}$ (giving us the expression at $c_2$).

Without any additional assumption  verification of the fact that \eqref{GenSol Hom Dirac} is a solution of \eqref{HomDirac} can be done by the direct substitution. The solution given by the formula \eqref{GenSol Hom Dirac} is general since the expressions at $c_1$ and at $c_2$ are linearly independent which can be seen from their values at $x=x_0$.
\end{proof}

\begin{lemma}\label{Lemma FP estimates}
Under the conditions of Theorem \ref{Thm DiracSPPS} let us define
\begin{align}
    c:= & \max\bigl\{ \|X^{(0)}\|,  \|Y^{(0)}\|, \|\widetilde X^{(0)}\|, \|\widetilde Y^{(0)}\| \bigr\}, & c_1:= & \max\bigl\{ \|f^2r_{11}+g^2r_{21}\|,  \|f^2r_{12}+g^2r_{22}\| \bigr\}, \label{Eq Estimates C and C1}\\
    c_2:= & \max\left\{ \left\|\frac{p_1}{g^2}\right\|, \left\|\frac{p_2}{f^2}\right\| \right\}, & c_3:= & \max\left\{ \left\|r_{11}\frac{f}{g}\right\|, \|r_{12}\|, \|r_{21}\|, \left\|r_{22}\frac{g}{f}\right\|\right\}, \label{Eq Estimates C2 and C3}
\end{align}
where $\|\cdot\|$ denotes max-norm on $[a,b]$. Then the following estimates hold for the functions $X^{(n)}$, $Y^{(n)}$, $\widetilde X^{(n)}$, $\widetilde Y^{(n)}$, $Z^{(n)}$, $\widetilde Z^{(n)}$, $n\ge 0$.
\begin{gather}
    \max\{|X^{(n)}(x)|, |Y^{(n)}(x)|, |\widetilde X^{(n)}(x)|, |\widetilde Y^{(n)}(x)|\}  \le c\cdot 2^n n! \sum_{k=0}^n\binom{n}{k}(c_1c_2)^k c_3^{n-k}\frac{|x-x_0|^{n+k}}{(n+k)!},\label{Eq Estimate Xn}\\
    \max\{|Z^{(n)}(x)|, |\widetilde Z^{(n)}(x)|\}  \le cc_1\cdot 2^{n+1} n! \sum_{k=0}^n\binom{n}{k}(c_1c_2)^k c_3^{n-k}\frac{|x-x_0|^{n+k+1}}{(n+k+1)!}.\label{Eq Estimate Zn}
\end{gather}
\end{lemma}

\begin{proof}
The proof is straightforward by induction. For $n=0$, \eqref{Eq Estimate Xn} follows directly from the definition of the constant $c$. Suppose \eqref{Eq Estimate Xn} is true for some $n$. Then assuming $x\ge x_0$ we obtain from \eqref{Zn} and \eqref{Eq Estimate Zn}
\begin{equation*}
    \begin{split}
       |Z^{(n)}(x)| & \le \int_{x_0}^x \left( 2c_1\cdot c\cdot 2^nn! \sum_{k=0}^n\binom{n}{k}(c_1c_2)^k c_3^{n-k}\frac{(t-x_0)^{n+k}}{(n+k)!}\right)\,dt \\
         & = cc_1\cdot 2^{n+1} n! \sum_{k=0}^n\binom{n}{k}(c_1c_2)^k c_3^{n-k}\frac{|x-x_0|^{n+k+1}}{(n+k+1)!}
     \end{split}
\end{equation*}
and from \eqref{Xn} and \eqref{Eq Estimate Xn} we obtain
    \begin{align*}
       |X^{(n)}(x)| & \le 2^{n+1}(n+1)! c\int_{x_0}^x \left( c_3 \sum_{k=0}^n\binom{n}{k}(c_1c_2)^k c_3^{n-k}\frac{(t-x_0)^{n+k}}{(n+k)!} \right.
       \displaybreak[2]
       \\
       &\quad + \left. c_1c_2 \sum_{k=0}^n\binom{n}{k}(c_1c_2)^k c_3^{n-k}\frac{(t-x_0)^{n+k+1}}{(n+k+1)!}\right)\,dt\\
       & = c(n+1)!2^{n+1} \sum_{k=0}^{n+1}\binom{n+1}{k}(c_1c_2)^k c_3^{n+1-k}\frac{|x-x_0|^{n+k+1}}{(n+k+1)!}.
     \end{align*}
The proof for the case $x\le x_0$ and for the functions $Y^{(n)}$, $\widetilde X^{(n)}$, $\widetilde Y^{(n)}$, $\widetilde Z^{(n)}$ is similar.
\end{proof}

\begin{corollary}\label{Corr FP estimates}
Under the conditions of Lemma \ref{Lemma FP estimates} the following estimates hold for any $n\ge 0$
\begin{gather}
    \max\{|X^{(n)}(x)|, |Y^{(n)}(x)|, |\widetilde X^{(n)}(x)|, |\widetilde Y^{(n)}(x)|\}  \le c\cdot 2^n |x-x_0|^n\bigl( c_1c_2|x-x_0|+c_3\bigr)^n,\label{Eq SimpleEstimate Xn} \\
    \max\{|Z^{(n)}(x)|, |\widetilde Z^{(n)}(x)|\}  \le cc_1\cdot 2^{n+1} |x-x_0|^{n+1}\bigl( c_1c_2|x-x_0|+c_3\bigr)^n.\label{Eq SimpleEstimate Zn}
\end{gather}
\end{corollary}

\begin{proof}
The result follows from Lemma \ref{Lemma FP estimates}. Indeed, we have, for example, for the function $X^{(n)}$
\begin{equation*}
    \begin{split}
       |X^{(n)}(x)| & \le c\cdot 2^nn! \sum_{k=0}^n\binom{n}{k}(c_1c_2)^k c_3^{n-k}\frac{|x-x_0|^{n+k}}{(n+k)!} \\
         & \le c\cdot 2^nn! |x-x_0|^n \sum_{k=0}^n\binom{n}{k}(c_1c_2)^k c_3^{n-k}\frac{|x-x_0|^{k}}{n!}\\
         & = c\cdot 2^n |x-x_0|^n \bigl( c_1c_2|x-x_0|+c_3\bigr)^n.
     \end{split}
\end{equation*}
\end{proof}

Now we present the proof of Theorem \ref{Thm DiracSPPS}.
\begin{proof}[Proof of Theorem \ref{Thm DiracSPPS}]
It follows from Corollary \ref{Corr FP estimates} that both series in \eqref{DiracSPPS} converge uniformly on $[a,b]$ as well as the series of  termwise derivatives, hence it is possible to apply the Dirac operator $B\frac{d}{dx}+P$ termwise to the series. Consider the second series in \eqref{DiracSPPS}. As it follows from the definitions \eqref{X0}--\eqref{Yn} and Lemmas \ref{Lemma NonHom Dirac} and \ref{Lemma Hom Dirac}, the functions $X^{(n)}$, $Y^{(n)}$, $n\ge 0$ satisfy
\begin{equation*}
     B\frac{d}{dx}\begin{pmatrix}
     fX^{(n)}\\
     gY^{(n)}
     \end{pmatrix}
      + P(x)\begin{pmatrix}
     fX^{(n)}\\
     gY^{(n)}
     \end{pmatrix} = n\cdot R(x)\begin{pmatrix}
     fX^{(n-1)}\\
     gY^{(n-1)}
     \end{pmatrix},
\end{equation*}
where for $n=0$ the symbols $X^{(-1)}$ and $Y^{(-1)}$ appear only to unify the notation. Therefore,
\begin{equation*}
\left(B\frac{d}{dx}+P\right) \begin{pmatrix}
    u_2 \\
    v_2 \\
\end{pmatrix}= \sum_{n=1}^\infty \frac{\lambda^n}{n!}\cdot n R\begin{pmatrix}
    f X^{(n-1)} \\
    g Y^{(n-1)} \\
\end{pmatrix} = \lambda R\sum_{n=1}^\infty \frac{\lambda^{n-1}}{(n-1)!}\begin{pmatrix}
    f X^{(n-1)} \\
    g Y^{(n-1)} \\
\end{pmatrix}=\lambda R\begin{pmatrix}
    u_2 \\
    v_2 \\
\end{pmatrix}.
\end{equation*}
That is, the second series in \eqref{DiracSPPS} is a solution of the system \eqref{GenDiracMatrix}. The proof for the first series in \eqref{DiracSPPS} is the same.

The linear independence of the solutions $(u_1,v_1)^T$ and $(u_2,v_2)^T$ can be obtained considering their values at $x=x_0$. Indeed, as it follows from the definitions \eqref{X0}--\eqref{Yt0},
\begin{equation}\label{DiracSPPS IC}
    Y_1(x_0)=\begin{pmatrix}
    u_1(x_0) \\
    v_1(x_0) \\
\end{pmatrix} =
\begin{pmatrix}
    f(x_0) \\
    0 \\
\end{pmatrix}
\qquad\text{and}\qquad Y_2(x_0)=\begin{pmatrix}
    u_2(x_0) \\
    v_2(x_0) \\
\end{pmatrix} =
\begin{pmatrix}
    0 \\
    g(x_0) \\
\end{pmatrix}.
\end{equation}
\end{proof}

\begin{remark}
Taking into account the results of Example \ref{ExampleSL} one can see that Theorem \ref{Thm DiracSPPS} contains the SPPS representation from \cite{KrPorter2010} as a particular case.
\end{remark}

\subsection{Construction of a non-vanishing particular solution}\label{SubSectPartSol}
In this subsection we show how the general solution of the homogeneous system \eqref{HomDirac} can be obtained from Theorem \ref{Thm DiracSPPS} and discuss how one can select a solution $(u,v)^T$ of \eqref{HomDirac} such that both functions $u$ and $v$ are non-vanishing on $[a,b]$.

The system \eqref{HomDirac} can be rewritten either as
\begin{equation}\label{System1HomDirac}
    \begin{cases}
        v' = -p_1 u - q v ,\\
        -u'  = -q u - p_2 v,
    \end{cases}
\end{equation}
or as
\begin{equation}\label{System2HomDirac}
    \begin{cases}
        v'+ q v = -p_1 u ,\\
        -u'+q u  = - p_2 v.
    \end{cases}
\end{equation}
Both systems can be considered as particular cases of a system of the type \eqref{GenDirac} taking $\lambda=-1$ and corresponding matrices $P$ and $R$. The homogeneous system associated to \eqref{System1HomDirac} possesses a non-vanishing particular solution $(u, v)^T = (1, 1)^T$, while the homogeneous system associated to \eqref{System2HomDirac} possesses a non-vanishing particular solution $(u, v)^T = (\exp(\int q(s)\,ds), \exp(-\int q(s)\,ds))^T$. Hence, the general  solution of the homogeneous system \eqref{HomDirac} can be obtained from Theorem \ref{Thm DiracSPPS} applied either to the system \eqref{System1HomDirac} or to the system \eqref{System2HomDirac}.

In the case when all the coefficients $p_{1,2}$ and $q$ are real-valued functions, it is possible to construct a non-vanishing solution (complex-valued) explicitly. The following well-known proposition can be used.

\begin{proposition}\label{Prop NonVanishingSol1}
Let the coefficients $p_{1,2}$ and $q$ of the system \eqref{HomDirac} are real-valued functions and $(u_1,v_1)^T$ and $(u_2,v_2)^T$ are two linearly independent real-valued solutions of the system \eqref{HomDirac}. Then the linear combination
\begin{equation*}
    \begin{pmatrix}
        u\\
        v
    \end{pmatrix} =
    \begin{pmatrix}
        u_1\\
        v_1
    \end{pmatrix} +
    i\begin{pmatrix}
        u_2\\
        v_2
    \end{pmatrix}
\end{equation*}
is a non-vanishing solution of the system \eqref{HomDirac}, i.e., both functions $u$ and $v$ do not have zeros on $[a,b]$.
\end{proposition}

\begin{proof}
Suppose that at some point $x_0$ we have $u(x_0) = 0$. Since both functions $u_1$ and $u_2$ are real valued, $u_1(x_0)=u_2(x_0)=0$ which means that at the point $x_0$  $(u_1(x_0),v_1(x_0))^T=(0,v_1(x_0))^T = c\cdot (0,v_2(x_0))^T=c\cdot (u_2(x_0),v_2(x_0))^T$ for some constant $c$, a contradiction with the linear independency of the solutions.
\end{proof}

In the general situation when one or several of the coefficients of the system \eqref{HomDirac} may possess complex values, we are not aware of any explicit method of constructing a non-vanishing solution. However the situation is not that bad, at least a non-vanishing solution always exists and there are plenty of them.

\begin{proposition}\label{Prop NonVanishingSol2}
Let $(u_1,v_1)^T$ and $(u_2,v_2)^T$ be two linearly independent solutions of \eqref{HomDirac}. Then there exists a linear combination
\begin{equation*}
    \begin{pmatrix}
        u\\
        v
    \end{pmatrix} =
    c_1\begin{pmatrix}
        u_1\\
        v_1
    \end{pmatrix} +
    c_2\begin{pmatrix}
        u_2\\
        v_2
    \end{pmatrix}
\end{equation*}
such that both functions $u$ and $v$ are non-vanishing on $[a,b]$.
\end{proposition}
\begin{proof}
The proof is based on the Sard's theorem and follows the proofs of Proposition 2.2 and Corollary 2.3 from \cite{CamporesiScala2011},
see also \cite[Remark 5]{KrPorter2010}.
Consider the complex projective space $\mathbb{CP}^1$, i.e., the quotient of $\mathbb{C}^2\setminus \{0\}$ by the action of $\mathbb{C}^*$. Proposition 2.2 from \cite{CamporesiScala2011} states the following. Let $I\subset \mathbb{R}$ be an interval. A differentiable map $f:I\to \mathbb{CP}^1$ cannot be surjective. And the proof is that the Sard's theorem implies that the image $f(I)$ has measure zero.

In Corollary 2.3 \cite{CamporesiScala2011} for two differentiable functions $y_1$ and $y_2$ which do not vanish simultaneously, authors consider a differentiable map
\begin{equation}\label{SardMap}
f:I\to \mathbb{CP}^1,\qquad x\mapsto f(x)=[y_1(x):y_2(x)].
\end{equation}
If a linear combination $c_1y_1 + c_2y_2$ vanishes at some point $x_0\in I$, then the determinant $\begin{vmatrix}  y_1(x_0) & -c_2 \\                                        y_2(x_0) & c_1\end{vmatrix}$ is equal to zero, which implies that $(y_1(x_0), y_2(x_0))$ is proportional to $(-c_2, c_1)$, hence $[-c_2:c_1]$ belongs to the image $f(I)$. Since the image $f(I)$ has a measure zero, the set of complex constants $[c_1:c_2]\in\mathbb{CP}^1$ for which the linear combination $c_1y_1 + c_2y_2$ vanishes at some point, has measure zero.

Now we consider $y_1=u_1$, $y_2=u_2$ and the corresponding map $f_1$ given by \eqref{SardMap}. As was mentioned in the proof of Proposition \ref{Prop NonVanishingSol1}, the functions $u_1$ and $u_2$ do not vanish simultaneously, hence the image $f_1([a,b])$ has measure zero. The same reasoning applies to $y_1=v_1$, $y_2=v_2$ and the corresponding map $f_2$. The union $f_1([a,b]) \cup f_2([a,b])$ also has measure zero, hence for almost all points $[c_1:c_2]\in\mathbb{CP}^1$ both linear combinations $c_1u_1+c_2u_2$ and $c_1v_1+c_2v_2$ do not vanish at any point of the segment $[a,b]$.
\end{proof}

\begin{remark}\label{RemMeasureZero}
As can be seen from the proof of Proposition \ref{Prop NonVanishingSol2}, one can obtain a non-vanishing solution of \eqref{HomDirac} by taking two linearly independent solutions $(u_1,v_1)^T$ and $(u_2,v_2)^T$  of \eqref{HomDirac}, choosing randomly a point on $\mathbb{CP}^1$, i.e., some complex numbers $c_1$ and $c_2$, and verifying if the linear combination $c_1(u_1,v_1)^T+c_2(u_2,v_2)^T$ vanishes on $[a,b]$. If not, we are done. If yes, repeating the process. Since the set of ``bad'' coefficients has measure zero, the non-vanishing solution in most cases will be obtained on the first try.
\end{remark}

\subsection{Spectral shift}\label{SubSectSpectralShift}
Similarly to the Taylor series, an approximation of the solution given by a truncation of the series \eqref{DiracSPPS} is more accurate near the origin, while the accuracy deteriorates as the absolute value of the parameter $\lambda$ increases. In \cite{CKT2013, KKB2013, KrPorter2010, KTV} the spectral shift technique was introduced and successfully applied to improve the accuracy of the approximation for the large $\lambda$.

Under the assumption that the matrix $R$ is symmetric, i.e., $r_{12}\equiv r_{21}$ or equivalently
\begin{equation}\label{TraceCondition}
    \operatorname{tr}BR(x)\equiv 0,\qquad x\in [a,b],
\end{equation}
application of the spectral shift is straightforward. Let a non-vanishing particular solution $(u,v)^T$ of \eqref{GenDiracMatrix} is known for some $\lambda=\lambda_0$. We can rewrite the system as
\begin{equation}\label{GenDirac SpShift}
    B\frac{dY}{dx} + \bigl( P(x)-\lambda_0 R(x)\bigr) Y = (\lambda-\lambda_0) R(x) Y,
\end{equation}
which is again a system of the type \eqref{GenDirac} due to the assumption \eqref{TraceCondition}, and $(u,v)^T$ is a non-vanishing solution of the system \eqref{GenDirac SpShift} corresponding to $\lambda-\lambda_0=0$. Hence one can construct the formal powers and obtain the general solution of \eqref{GenDirac SpShift} (which also is a general solution of \eqref{GenDirac}) as the series with respect to the spectral parameter $\Lambda:=\lambda-\lambda_0$ using Theorem \ref{Thm DiracSPPS}.

For the general case, when the condition \eqref{TraceCondition} is not satisfied, we introduce the new unknown vector function $U$ by
\[
Y = w(x)U,\qquad w(x) = \exp\left(-\frac{\lambda_0}2 \int\operatorname{tr} BR(s)\,ds\right).
\]
The system \eqref{GenDirac SpShift} for the new unknown $U$ takes the form
\begin{equation}\label{GenDirac_SpShift2}
    B\frac{dU}{dx} + \left( P(x)-\lambda_0 R(x)-\frac{\lambda_0}2 B\operatorname{tr}BR(x)\right) U= (\lambda-\lambda_0) R(x) U,
\end{equation}
a system which is again of the form \eqref{GenDirac}. If $(u,v)^T$ is a non-vanishing solution for \eqref{GenDirac} corresponding to $\lambda=\lambda_0$, one can take $\frac 1{w(x)}(u,v)^T$ as a non-vanishing solution for \eqref{GenDirac_SpShift2} corresponding to $\lambda-\lambda_0=0$ and construct the SPPS representation with respect to $\Lambda=\lambda-\lambda_0$.

\subsection{Discontinuous coefficients}\label{SubSectDiscontinuous}
Following \cite[Chapter 1]{Zettl} we can consider the system \eqref{GenDirac} with the coefficients $p_i$, $q$, $r_{ij}\in L^1(a,b)$. In such case a vector function $(u,v)^T$ is called a solution of \eqref{GenDirac} if both functions $u,v$ are absolutely continuous on $[a,b]$ and satisfy the system almost everywhere. With the slight modification to the proof the SPPS representation is valid for this case.

Suppose that functions $f,g\in\operatorname{AC}[a,b]$ are such that the following assumption holds.
\begin{equation}\label{AssumIntegrability}
    \left\{\frac{p_2}{f^2},\frac{p_1}{g^2},r_{11}\frac fg,r_{22}\frac gf\right\}\subset L^1(a,b).
\end{equation}
Note that it is sufficient for $f$ and $g$ to be non-vanishing on $[a,b]$. Then one can define the systems of functions $X^{(n)}$, $Y^{(n)}$, $Z^{(n)}$, $\widetilde X^{(n)}$, $\widetilde Y^{(n)}$  and $\widetilde Z^{(n)}$ by \eqref{X0}--\eqref{Yt0}.
\begin{lemma}\label{Lemma FP estimates discont}
Let $f$ and $g$ be absolutely continuous functions satisfying \eqref{AssumIntegrability}. Define (to simplify the formulas we assume that $x>x_0$)
\begin{align*}
h(x) & := \max\left\{ \left|\frac{p_2(x)}{f^2(x)}\right|,\left|\frac{p_1(x)}{g^2(x)}\right|,\left|r_{11}(x)\frac{f(x)}{g(x)}\right|,
\left|r_{22}(x)\frac{g(x)}{f(x)}\right|,|r_{12}(x)|, |r_{21}(x)|\right\},\\
F(x)&:= |f(x_0)g(x_0)|+\int_{x_0}^x h(s)\,ds,\\
G(x)&:= \int_{x_0}^x \Bigl(\bigl| f^2(s)r_{11}(s) + g^2(s) r_{21}(s)\bigr|+
\bigl| f^2(s)r_{12}(s) + g^2(s) r_{22}(s)\bigr|\Bigr)\,ds.
\end{align*}
Then the functions $X^{(n)}$, $Y^{(n)}$, $Z^{(n)}$, $\widetilde X^{(n)}$, $\widetilde Y^{(n)}$  and $\widetilde Z^{(n)}$, $n\ge 0$, are absolutely continuous and satisfy the following estimates
\begin{gather*}
    \max\left\{|X^{(n)}(x)|, |Y^{(n)}(x)|, |\widetilde X^{(n)}(x)|, |\widetilde Y^{(n)}(x)|\right\}  \le \dfrac{(F(x))^{n+1}}{(n+1)!} \sum_{k=0}^n\binom{n}{k}\dfrac{(G(x))^{n-k}}{(n-k)!},\\
    \max\left\{|Z^{(n)}(x)|, |\widetilde Z^{(n)}(x)|\right\}  \le \dfrac{(F(x))^{n}}{n!} \sum_{k=0}^{n-1}\binom{n-1}{k}\dfrac{(G(x))^{n-k}}{(n-k)!}.
\end{gather*}
\end{lemma}

The proof is by induction, similarly to the proof of Proposition 5 from \cite{BCK2015}. We left the details to the reader.

Recall that the space $AC[a,b]$ of absolutely continuous functions coincides with the Sobolev space $W_1^1[a,b]$, hence the series $\sum_{n=0}^\infty u_n(x)$, where $u_n\in AC[a,b]$, converges to an absolutely continuous function if the series converges at some point $x_0\in[a,b]$ and the series of the derivatives $\sum_{n=0}^\infty u_n'$ converges in $L^1(a,b)$ norm.

Now using estimates from Lemma \ref{Lemma FP estimates discont} one can easily verify that the series $\sum_{n=0}^\infty \lambda^n X^{(n)}$, $\sum_{n=0}^\infty \lambda^n Y^{(n)}$ and $\sum_{n=0}^\infty \lambda^n Z^{(n)}$ are absolutely continuous functions and that Theorem \ref{Thm DiracSPPS} holds with the following change: $f$ and $g$ are required to be absolutely continuous functions satisfying \eqref{AssumIntegrability}.

\section{General linear system}\label{SectionSPPSsystem}
Consider a general linear system of two first order differential equations
\begin{equation}\label{GenLinSys}
    \mathcal{P}(x)\frac{dY}{dx}+\mathcal{Q}(x)Y = \lambda \mathcal{R}(x)Y,\qquad x\in [a,b],
\end{equation}
where $Y=(y_1,y_2)^T$  is the unknown vector-function and $\mathcal{P}$, $\mathcal{Q}$, $\mathcal{R}$ are $2\times 2$ matrices whose entries are continuous complex-valued functions. Assume additionally that $\det \mathcal{P}\ne 0$ for all $x\in [a,b]$.

Multiplying \eqref{GenLinSys} by $B\mathcal{P}^{-1}$ we arrive to the system
\begin{equation}\label{GenLinSys2}
    B\frac{dY}{dx} + Q(x) Y = \lambda R(x)Y,
\end{equation}
where $Q=B\mathcal{P}^{-1}\mathcal{Q}$ and $R=B\mathcal{P}^{-1}\mathcal{R}$. In general, the system \eqref{GenLinSys2} is not of the type \eqref{GenDiracMatrix} since the condition $\operatorname{tr} BQ(x)\equiv 0$ need not be satisfied. So we may proceed similarly to Subsection \ref{SubSectSpectralShift} and introduce new unknown vector function $U$ by
\[
Y = w(x)U,\qquad w = \exp\left(\frac 12 \int\operatorname{tr}BQ(s)\,ds\right),
\]
for which the system \eqref{GenLinSys2} takes the form
\begin{equation}\label{GenLinSys3}
    B\frac{dU}{dx}+\left(Q(x)+\frac 12 B\operatorname{tr}BQ(x)\right)U=\lambda R(x)U,
\end{equation}
and since $B^2=-I$ one easily checks that
\[
\operatorname{tr}\left(BQ(x)+\frac 12 B^2\operatorname{tr} BQ(x)\right)=\operatorname{tr}BQ(x) - \frac 12 \operatorname{tr}BQ(x)\cdot \operatorname{tr}I=0,
\]
hence the system \eqref{GenLinSys3} is of the type \eqref{GenDiracMatrix} and Theorem \ref{Thm DiracSPPS} can be applied to it.
\section{Numerical illustration}\label{Section Numerics}
\subsection{General scheme and implementation details}\label{SubSectGenScheme}
The general scheme of application of the SPPS representation to the approximate solution of initial value and spectral problems for the system \eqref{GenDirac} is similar to that for the SPPS representation for the Sturm-Liouville equation, see \cite{KrPorter2010}, \cite{KrTNewSPPS}, \cite{KTV}. Consider an initial value problem
\begin{equation}\label{IVproblem}
    Y(a) = \begin{pmatrix}
    y_1\\
    y_2
    \end{pmatrix}
\end{equation}
and a spectral problem given by the following boundary conditions
\begin{equation}
 (a_1, a_2) Y(a) =0, \qquad (b_1, b_2) Y(b) =0, \label{BV}
\end{equation}
where $y_1$, $y_2$, $a_1$, $a_2$, $b_1$, $b_2$ are some complex constants satisfying $|a_1|+|a_2|\ne 0$ and  $|b_1|+|b_2|\ne 0$. We would like to stress out that more complicated problems like having mixed or spectral parameter dependent boundary conditions can be treated as well.

\begin{enumerate}
\item Construct a non-vanishing particular solution of \eqref{HomDirac} as described in Subsection \ref{SubSectPartSol}. According to Remark \ref{RemMeasureZero} one can take two linearly independent solutions $(u_1,v_1)^T$ and $(u_2,v_2)^T$ satisfying $u_1(x_0)\ne 0$, take some finite set of complex constants $c_1,\ldots,c_K$ (they can be taken, e.g., by randomly choosing modulus $\rho_k$ and phase $\varphi_k$ in the polar representation $c_k=\rho_k e^{i\varphi_k}$) and select as the non-vanishing solution the one having the least value of the expression
    \begin{equation}\label{EqSelectNonVanishing}
        \max\left\{ \frac{\max_{x\in[a,b]}|u_1(x)+c_k u_2(x)|}{\min_{x\in[a,b]}|u_1(x)+c_k u_2(x)|}, \frac{\max_{x\in[a,b]}|v_1(x)+c_k v_2(x)|}{\min_{x\in[a,b]}|v_1(x)+c_k v_2(x)|}\right\}.
    \end{equation}
\item Calculate the formal powers $X^{(n)}$, $Y^{(n)}$, $\widetilde X^{(n)}$ and $\widetilde Y^{(n)}$, $n=0,\ldots,N$ according to \eqref{X0}--\eqref{Yt0}. The number $N$ may be estimated either using the bounds from Lemma \ref{Lemma FP estimates} or simply by verifying that
    \[
    \frac{1}{N!}\max\{\|X^{(N)}\|,\|\widetilde X^{(N)}\|,\|Y^{(N)}\|,\|\widetilde Y^{(N)}\|\}
    \]
    is sufficiently small. E.g., it is equal to zero in machine-precision arithmetic. Here $\|\cdot\|$ denotes max-norm on $[a,b]$.
\item
For the solution of the initial value problem \eqref{IVproblem} one calculates approximate solutions $\widetilde Y_1$ and $\widetilde Y_2$ using truncated sums from \eqref{DiracSPPS}.

In the particular case when $x_0=a$ the solution of  \eqref{IVproblem} due to \eqref{SPPS IC} is given by
\[
\widetilde Y = \frac{y_1}{f(a)} \widetilde Y_1 + \frac{y_2}{g(a)}\widetilde Y_2.
\]
In the general case, one finds the constants $c_1$ and $c_2$ in \eqref{DiracSPPSgen} by solving the linear system of equations
    \[
    c_1 \widetilde Y_1(a) + c_2 \widetilde Y_2(a) = \begin{pmatrix}
    y_1\\
    y_2
    \end{pmatrix}.
    \]
\item For the solution of the spectral problem \eqref{BV} in the particular case $x_0=a$ note that a solution satisfying the first boundary condition in \eqref{BV} is given (up to a multiplicative constant) by
    \[
    Y = \frac{a_2}{f(a)} Y_1 - \frac{a_1}{g(a)}Y_2.
    \]
    This solution satisfies the second boundary condition in \eqref{BV} iff.\ the following characteristic equation is satisfied:
    \[
    \Xi(\lambda) := \frac{a_2 b_1}{f(a)}u_1(b; \lambda) - \frac{a_1 b_1}{g(a)}u_2(b; \lambda) +
    \frac{a_2 b_2}{f(a)}v_1(b; \lambda) - \frac{a_1 b_2}{g(a)}v_2(b; \lambda) = 0.
    \]
    The function $\Xi$, known as characteristic function of the spectral problem, is analytic. By truncating the series representations \eqref{DiracSPPS} for the functions $u_1$, $u_2$, $v_1$ and $v_2$ one obtains approximate characteristic function which results to be a polynomial. Its (complex) roots closest to the origin approximate the exact eigenvalues, while more distant roots result to be spurious. See \cite[Section 7.2]{CKT2013} for further discussion on how these spurious roots can be discarded.

    For the general case one considers the general solution given by \eqref{DiracSPPSgen} and substitutes it into boundary conditions \eqref{BV}. The existence of non-trivial solution is equivalent to the following condition:
    \begin{equation}\label{DetCondition}
        \det \begin{pmatrix}
        a_1u_1(a;\lambda) + a_2 v_1(a;\lambda) & a_1u_2(a;\lambda) + a_2 v_2(a;\lambda)\\
        b_1u_1(b;\lambda) + b_2 v_1(b;\lambda) & b_1u_2(b;\lambda) + b_2 v_2(b;\lambda)
    \end{pmatrix} = 0.
    \end{equation}
    Truncating the series representations for $u_1$, $u_2$, $v_1$ and $v_2$ in \eqref{DetCondition} on obtains a polynomial approximating the characteristic function of the problem.
\item If more than few closest to zero eigenvalues are needed, one may apply several spectral shift procedures described in Subsection \ref{SubSectSpectralShift}. Since for the one-dimensional Dirac system the distance between consequent eigenvalues remains bounded for all eigenvalues, see \cite[Chap. 7, \S2]{LevitanSargsjan}, the following simple recipe showed to deliver excellent results. We use the SPPS representation to find the eigenvalues $\lambda_0$ and $\lambda_{\pm1}$ and general solutions corresponding to $\lambda_{\pm 1}$. Than we look for complex coefficients which give us non-vanishing solutions corresponding to $\lambda_{\pm 1}$ using the same criteria as in \eqref{EqSelectNonVanishing}. Now we use $\lambda_1$ and $\lambda_{-1}$ as centers for the spectral shift procedure and use the corresponding non-vanishing solutions to calculate the formal powers. Resulting SPPS representations give us $\lambda_2$ and $\lambda_{-2}$ and corresponding general solutions. And so on, having found $\lambda_n$ and $\lambda_{-n}$ and corresponding non-vanishing solutions we use them as the new centers for the spectral shift procedure (the total step being $\lambda_n-\lambda_{n-1}$ and $\lambda_{-n}-\lambda_{-(n-1)}$, bounded quantity as $n\to\infty$) until required number of eigenvalues be find.
\end{enumerate}
We would like to emphasize that all steps of the proposed scheme can be realized numerically, there is no need for the integrals in \eqref{X0}--\eqref{Yt0} to be available in the closed form. We refer the reader to \cite{CKT2013}, \cite{KrTNewSPPS} and \cite{KT AnalyticApprox}  for additional details and only mention that in the following examples all the functions involved were represented by their values on the uniform mesh and  Newton-Cotes 6 point rule was used for indefinite integration. All computations were performed in double machine precision in Matlab 2017.

\subsection{Example: spectral problem for a Dirac system}
Consider the following spectral problem (Example 3.4 from \cite{AnnabyTharwat})
\begin{equation}\label{SysAnnabyTharwat}
\begin{cases}
v'-xu =\lambda u,\\
-u'+v = \lambda v,
\end{cases}\qquad 0\le x\le 1
\end{equation}
with the boundary conditions
\begin{equation}\label{ExAnnabyTharwatBC}
    u(0) = u(1) = 0.
\end{equation}
The characteristic function for this problem can be written in the terms of Airy functions, see \cite{AnnabyTharwat},
allowing one to compute exact eigenvalues with any desired precision using, e.g., Wolfram Mathematica.

\begin{table}[htb]
\centering
\begin{tabular}{ccccc}
\hline
$n$ & $\lambda_n$ & Abs.\ error, & Abs.\ error, & Abs.\ error, \\
& & directly from \eqref{DiracSPPS} & using spectral shift & reported in \cite{AnnabyTharwat}\\
\hline
-100 & -313.9101939150852 & & $5.3\cdot 10^{-6}$ &\\
-50 & -156.8314900718780 & & $9.6\cdot 10^{-8}$ &\\
-20 & -62.58649828127890 & & $3.8\cdot 10^{-10}$ &\\
-10 & -31.17522014114365 & $1.8\cdot 10^{-3}$ & $5.1\cdot 10^{-12}$ &\\
-7 & -21.75442521496494 & $1.7\cdot 10^{-7}$ & $5.1\cdot 10^{-13}$ &\\
-5 & -15.47654249528427 & $6.4\cdot 10^{-10}$ & $6.0\cdot 10^{-14}$ &\\
-2 & -6.079080595285440 & $3.5\cdot 10^{-13}$ & $1.4\cdot 10^{-14}$ & $3.5\cdot 10^{-13}$ \\
-1 & -2.977189710951455 & $5.1\cdot 10^{-14}$ & $5.1\cdot 10^{-14}$ & $2.6\cdot 10^{-13}$ \\
0 & 1 & $1.9\cdot 10^{-14}$ & $1.9\cdot 10^{-14}$ & $3.1\cdot 10^{-13}$ \\
1 & 3.478833069692201 & $1.4\cdot 10^{-14}$ & $1.4\cdot 10^{-14}$ & $1.2\cdot 10^{-12}$ \\
2 & 6.578592238156064 & $3.1\cdot 10^{-13}$ & $1.2\cdot 10^{-14}$ & \\
5 & 15.97642352997195 & $1.1\cdot 10^{-9}$ & $5.1\cdot 10^{-14}$ &\\
7 & 22.25436259528469 & $4.2\cdot 10^{-7}$ & $5.1\cdot 10^{-12}$ &\\
10 & 31.67518895778715 & $1.5\cdot 10^{-3}$ & $5.0\cdot 10^{-12}$ &\\
20 & 63.08649039551696 & & $3.7\cdot 10^{-10}$\\
50 & 157.3314888061299 & & $9.5\cdot 10^{-8}$\\
100 & 314.4101935985044 & & $5.2\cdot 10^{-6}$\\
\hline
\end{tabular}
\caption{Eigenvalues of the spectral problem \eqref{SysAnnabyTharwat}, \eqref{ExAnnabyTharwatBC} and absolute errors of the approximate eigenvalues obtained using the representation \eqref{DiracSPPS} truncated to $N=100$ terms, using additionally the spectral shift technique and of those reported in \cite{AnnabyTharwat}.}
\label{Ex1 Tab1}
\end{table}

We used 2001 points mesh to represent all the functions involved in this example and computed the formal powers for $n\le 100$. In Table \ref{Ex1 Tab1} we present the absolute errors of the approximate eigenvalues obtained either directly from the SPPS representation \eqref{DiracSPPS} or using the spectral shift procedure. In the latter case the eigenvalues $\lambda_n$ for $|n|\le 100$ were computed. Even direct application of the SPPS representation delivers more accurate values than those reported in \cite{AnnabyTharwat} requiring only 0.1sec of computation time, while the spectral shift technique allows one to obtain two hundreds eigenvalues with a good accuracy in about 20 seconds.

\subsection{Example: application to Sturm-Liouville spectral problems}
In Example \ref{ExampleSL} we showed how a Sturm-Liouville equation can be transformed into a one dimensional Dirac system. A spectral problem for equation \eqref{SLeq} can be transformed into a spectral problem for the system \eqref{SLequivDirac} as well. Indeed, consider a boundary condition of the form
\begin{equation}\label{BCforSL}
\alpha u(a) +\beta u'(a) = 0.
\end{equation}
Due to \eqref{EqSLtoDirac} we have that
\[
u'(a) = \omega \frac{v(a)}{p(a)} +\frac{u_0'(a)}{u_0(a)}u(a),
\]
hence \eqref{BCforSL} is equivalent to the following (spectral parameter dependent) boundary condition
\begin{equation}\label{EquivBCforDirac}
    \left(\alpha + \frac{\beta u_0'(a)}{u_0(a)}\right) u(a) + \frac{\omega \beta}{p(a)} v(a) = 0.
\end{equation}

The general scheme presented in Subsection \ref{SubSectGenScheme} can be applied to spectral parameter dependent boundary conditions of the form \eqref{EquivBCforDirac} with minimal modifications. Note that rewriting the equation $(p(x)u')' + q(x)u = 0$ as an equivalent system
\[
\begin{cases}
v' + q(x)u =0,\\
u' - \frac{1}{p(x)}v = 0,
\end{cases}
\]
one can apply the results of Subsection \ref{SubSectPartSol} to construct a non-vanishing particular solution $u_0$ of \eqref{SLeq} as well.

One possible advantage of reformulating a Sturm-Liouville spectral problem as an equivalent Dirac system and applying the proposed SPPS representation instead of the SPPS representation from \cite{KrPorter2010} consists in the following. The eigenvalues $\lambda_n=\omega_n^2$ of the Sturm-Liouville problem grow as $\frac{\pi}{b-a} n^2$ as $n\to\infty$. However the spectral shift technique works best if the change of the spectral parameter remains bounded on each step. As a result, obtaining $n$ eigenvalues requires $O(n^2)$ steps using the SPPS representation from \cite{KrPorter2010}. In contrary, Dirac-based approach requires only $O(n)$ steps greatly reducing computation time if one needs large number of eigenvalues. We would like to point out that the method proposed in \cite{KNT2017} is better suited for computing large sets of eigenvalues, nevertheless the SPPS representation is simpler and still being used for numerous applications, see, e.g., \cite{BR2016}, \cite{BR2017}, \cite{BRM2019}, \cite{LO2017}, \cite{RH2017}.

Consider the following spectral problem (the first Paine
problem, \cite{Paine}, see also \cite[Example 7.4]{KT AnalyticApprox}) to illustrate this advantage numerically.
\begin{equation}\label{EqPaineProblem}
\begin{cases}
-u^{\prime\prime}+e^{x}u=\lambda u, \quad 0\leq x\leq\pi,\\
u(0,\lambda)=u(\pi,\lambda)=0.
\end{cases}
\end{equation}

We computed approximate eigenvalues for this problem using the Dirac system reformulation as well as directly the SPPS representation from \cite{KrPorter2010}. For the latter we applied two different spectral shift strategies, to the nearest new eigenvalue and constant step size. The first one requires less steps but is known to fail eventually due to increasing gaps between consecutive eigenvalues, the second requires more steps but allowed us to obtain more accurate results previously, see \cite{CKT2013}, \cite{KrTNewSPPS}.

\begin{table}[htb]
\centering
\begin{tabular}{ccccc}
\hline
& & Abs.\ error, & Abs.\ error, & Abs.\ error, \\
$n$ & $\lambda_n$ & using the Dirac  & using spectral shift to  & using a constant step \\
& & system approach & the nearest eigenvalue & size spectral shift  \\
\hline
0 & 1.519865821099347 & $6.4\cdot 10^{-13}$ & $2.0\cdot 10^{-14}$ & $4.2\cdot 10^{-14}$ \\
1 & 4.943309822144690 & $4.0\cdot 10^{-13}$ &$6.8\cdot 10^{-14}$ & $6.3\cdot 10^{-13}$ \\
2 & 10.28466264508758 & $8.7\cdot 10^{-13}$ &$4.4\cdot 10^{-13}$ & $2.6\cdot 10^{-12}$ \\
3 & 17.55995774641423 & $1.3\cdot 10^{-12}$ &$3.8\cdot 10^{-13}$ & $3.5\cdot 10^{-12}$ \\
5 & 37.96442586193434 & $4.2\cdot 10^{-13}$ &$3.1\cdot 10^{-13}$ & $4.7\cdot 10^{-12}$ \\
10 & 123.4977068009282 & $3.9\cdot 10^{-12}$ &$1.4\cdot 10^{-12}$ & $4.5\cdot 10^{-12}$ \\
25 & 678.9217784771679 & $1.0\cdot 10^{-11}$ &$1.0\cdot 10^{-12}$ & $6.0\cdot 10^{-12}$ \\
50 & 2604.036332024594 & $1.4\cdot 10^{-11}$ &$1.4\cdot 10^{-11}$ & $3.9\cdot 10^{-11}$ \\
75 & 5779.062267233881 & $2.6\cdot 10^{-11}$ &$1.1\cdot 10^{-11}$ & $5.9\cdot 10^{-8}$ \\
100 & 10204.07191390758 & $7.1\cdot 10^{-11}$ &$1.1\cdot 10^{-7}$ & $3.8\cdot 10^{-5}$ \\
150 & 22804.07903279700 &  $1.3\cdot 10^{-10}$ & $1.1\cdot 10^{-2}$ & $2.4\cdot 10^{-3}$\\
\hline
\end{tabular}
\caption{Eigenvalues of the spectral problem \eqref{EqPaineProblem} and absolute errors of the approximate eigenvalues obtained converting the problem to a Dirac system and directly using the SPPS representation from \cite{KrPorter2010} with two different spectral shift strategies.}
\label{Ex2 Tab1}
\end{table}

We used 50001 points mesh to represent all the functions involved in this example and computed the formal powers for $n\le 100$. The spectral shift for the constant step was taken equal to $5$. Note that such large number of mesh points was taken in order to avoid integration errors and to be able to illustrate SPPS-related behavior, compare to \cite[Example 6.1]{KTV}. In Table \ref{Ex2 Tab1} we present the absolute errors of the approximate eigenvalues. As one can observe, the performance of all three approaches was comparable for the first 50 eigenvalues. For larger index eigenvalues the accuracy of the Dirac system based approach remains essentially the same while the accuracy of the eigenvalues obtained using directly the SPPS representation from \cite{KrPorter2010} started to deteriorate (surprisingly more rapidly when the uniform step size was used).

\section*{Acknowledgements}
The authors acknowledge the support from CONACYT, Mexico via the project 222478. N.~Guti\'{e}rrez Jim\'{e}nez would like to express his gratitude to the Mathematical department of Cinvestav where he completed the PhD program (the presented paper contains part of the obtained results) and to CONACYT, Mexico for the financial support making it possible.

\end{document}